\documentclass[10pt]{amsart}
\usepackage{amssymb,amsmath,amsthm,amsfonts,amsopn,url}
\usepackage{amscd,amssymb,amsopn,amsmath,amsthm,graphics,amsfonts,enumerate,verbatim,calc}
\usepackage[all]{xy}
\theoremstyle{plain}
\newtheorem{thm}{Theorem}[section]
\newtheorem*{mt*}{Main Theorem}
\newtheorem*{cj*}{Conjecture}
\newtheorem*{nt*}{Notations}

\newtheorem{prop}[thm]{Proposition}
\newtheorem{lemma}[thm]{Lemma}
\newtheorem{cor}{Corollary}
\newtheorem{rem}{Remark}
\newtheorem{example}{Example}[section]
\theoremstyle{definition}

\newcommand{\ideal}[1]{\mathfrak{#1}}

\newcommand{\m}{\ideal{m}}

\newcommand{\func}[1]{\mathrm{#1} \,}
\newcommand{\grade}{\func{grade}}

\newcommand{\height}{\func{ht}}

\newcommand{\Ass}{\func{Ass}}

\newcommand{\Ker}{\func{Ker}}

\newcommand{\NN}{{\mathbb N}}
\newcommand{\ZZ}{{\mathbb Z}}

\title[]{Examples of local Cohomology Modules for Ramified Regular Local Rings having Finite Set of Associated Primes}

\author[]{Rajsekhar Bhattacharyya}

\address{Dinabandhu Andrews College, Garia, Kolkata 700084, India}

\email{rbhattacharyya@gmail.com}

\keywords{Local Cohomology}

\subjclass[2010]{13D45}
\begin{document}

\begin{abstract}
Lyubeznik's conjecture, (\cite{Ly1}, Remark 3.7) asserts the finiteness of the set ssociated primes of local cohomology modules for regular rings. But, in the case of ramified regular local ring, it is open. Recently, in Theorem 1.2 of \cite{Nu1}, it is proved that in any Noetherian regular local ring $S$ and for a fixed ideal $J\subset S$, associated primes of local cohomology $H^i_J(S)$ for $i\geq 0$ is finite, if it does not contain $p$. In this paper, we use this result to construct examples of local cohomology modules for ramified regular local ring so that they have finitely many associated primes.
\end{abstract}

\maketitle

\section{introduction}

Consider a Noetherian ring $R$ and an $R$-module $M$. For an ideal $I\subset R$ and for some integer $i\geq 0$, let $H^i_I (M)$ be $i$-th local cohomology module with support in the ideal $I$. In the fourth problem of \cite{Hu}, it is asked that whether local cohomology modules of Noetherian rings have finitely many associated prime ideals. There are several important examples where we have the finiteness of associated primes of local cohomology modules: (1) For regular rings of prime characteristic \cite{HS}, (2) for regular local and affine rings of characteristic zero \cite{Ly1}, (3) for unramified regular local rings of mixed characteristic \cite{Ly2} and (4) for smooth algebra over $\ZZ$ and $\ZZ_{p\ZZ}$ \cite{BBLSZ}. All these results support Lyubeznik's conjecture, (\cite{Ly1}, Remark 3.7): 

\begin{cj*} Let $R$ be a regular ring and $I\subset R$ be its ideal, then for each $i\geq 0$, $i$-th local cohomology module $H^i_I (R)$ has finitely many associated prime ideals.
\end{cj*}

Recently, in \cite{Pu2}, it is shown that for excellent regular ring $R$ of dimension $d$, containing a field of characteristic zero, for an ideal $I\subset R$, $\Ass_R H^{d-1}_I (R)$ is finite. 

Lyubeznik's conjecture is open for the ramified regular local ring, but recently, there is a significant progress (see \cite{Nu1}). In fact, in Theorem 1.2 of \cite{Nu1}, it is proved that in a Noetherian regular local ring $S$ and for a fixed ideal $J\subset S$, associated primes of local cohomology $H^i_J(S)$ for $i\geq 0$ is finite, if it does not contain $p$. In this paper, we use this result to construct examples of local cohomology modules of ramified regular local ring so that they have finitely many associated primes.

In section 2, we review some of the basic results which we need in proving our main results. In section 3, we prove our main results. At first, in Proposition 3.1, we prove a new variant of the result of Theorem 2.1 of \cite{Bh}. Using Proposition 3.1, in Theorem 3.2, we prove our first main result: Consider a ramified regular local ring $S$ of mixed characteristic $p>0$ of dimension $d$ and choose a regular system of parameters such that $p$ is not in the ideal generated by any $d-1$ of them. Let $J$ be an ideal generated by monomials formed out of $d-1$ such a regular system of parameters. Consider an ideal $\beta$ in a polynomial rings in $d-1$ indeterminates over the prime field, generated by similar monomials but now formed out of $d-1$ indeterminates. If $\beta$ is perfect, then $\Ass_S H^i_J (S)$ is finite for every $i\geq 0$ except $i= \height{J}+1, \height{J}+2$. It is important to note that, here the monomials are generated by regular system of parameters, which is obviously a system of parameter, but it does not contain $p$ as one of the parameter.

As a consequence of Theorem 3.2, in Corollary 1, we have our next important result which can be stated as follows: Consider ramified regular local ring $S$ of dimension $d$ and let there be positive integers $m$ and $n$ such that $d-1= mn$. Consider a set of regular system of parameters where $p$ is not in the ideal generated by any $d-1$ of them. If $J$ is an ideal, generated by size $t$ minors of $m\times n$ matrix formed by $d-1$ such regular system of parameters, then for every $i\geq 0$, except $i= \height{J}+1, \height{J}+2$, $\Ass_{S} H^{i}_{J}(S)$ is finite. In Theorem 3.3, we present our third main result. There, we construct examples of ramified regular local rings, so that for a countable set of ideals, there are countable set of local cohomology modules with supports in those ideals, having finite set of associated primes.
 
%It is important to note that, here the monomials are generated by regular system of parameters, which is also a system of parameter, but does not contain $p$.
%then Here we use result of Theorem 1.2 of \cite{Nu1}. We have our examples as given in Theorem 3.2 (stae about it...) and in its Corollary 1(stae about it...) .
%Here we construct example of ramified ring in which for a large 

\section{basic results}

In this section, we discuss two basic results which we need to prove the results in the next section. For local cohomology modules we have the following well known result, regarding the behaviour of associated primes of local cohomology modules under faithfully flat extension.

\begin{lemma}
Let $S$ be faithfully flat over $R$. Then, for the local cohomology modules of an arbitrary $R$-module $M$, for every ideal $I\subset R$ and for $i\geq 0$, $\Ass_{S} H^i_{IS}(M\otimes_R S)$ is a finite set if and only if $\Ass_{R} H^i_I(M)$ is a finite set.
\end{lemma}

Next, we state the following known result (see Theorem 1.2 of \cite{Nu1}). We need it to prove Proposition 3.1 in the next section.

\begin{thm}
Let $S$ be a regular commutative Noetherian local ring of mixed characteristic $p > 0$. Then, the set of associated primes
of $H^i_J (S)$ that do not contain $p$ is finite for every $i\in \NN$ and every ideal $J\subset S$.
\end{thm}

\section{main results}

We begin the section with the following proposition which is a new variant of the result of Theorem 2.1 of \cite{Bh} to prove the next theorem. The part of the proof this proposition is similar to that of Theorem 2.1 of \cite{Bh}.  

\begin{prop}
Let $S$ be a ramified regular local ring of mixed characteristic $p>0$ and let $J\subset S$ be an ideal. Suppose there exists an unramified regular local ring $R$ with the following properties:

(1) $S$ is an Eisenstein extension of $R$ by Eisenstein polynomial $f(X)\in R[X]$. 

(2) $I\subset R[X]$ is an ideal such that $IS= J$ and multiplication by $p$ on $H^{i}_{I}(R[X])$ is an isomorphism. 

(3) $p\notin \Ass_{\ZZ} H^{i+1}_{I}(R[X])$

Then $p\notin \Ass_S H^i_J(S)$ and $\Ass_S H^i_J(S)$ is finite. 
\end{prop}

\begin{proof}
Here $p$ is an $R[X]$-regular element. So, for short exact sequence $$0\rightarrow R[X]\stackrel{p}{\rightarrow} R[X]\rightarrow R[X]/pR[X]\rightarrow 0$$ we get the following long exact sequence of local cohomologies $$..\rightarrow H^{i-1}_{I}(R[X]/pR[X])\stackrel{\delta}\rightarrow H^{i}_{I}(R[X])\stackrel{p}{\rightarrow}H^{i}_{I}(R[X])\rightarrow H^{i}_{I}(R[X]/pR[X]) \rightarrow ..$$ Since $p$ is a non zero divisor of both $H^i_{I} (R[X])$ and $H^{i+1}_{I}(R[X])$, from above long exact sequence we get that, $H^i_{I} (R[X]/pR[X])= H^i_{I}(R[X])/pH^i_{I} (R[X])$. Moreover, since multiplication by $p$ on $H^{i}_{I}(R[X])$ is an isomorphism, $H^i_{I} (R[X]/pR[X])= H^i_{I} (R[X])/pH^i_{I} (R[X])= 0$.

Here $S=R[X]/f(X)$ where $(R,\m)$ is an unramified regular local ring and Eisenstein polynomial $f(X)$ is of the form $f(X)=X^n+ a_1X^{n-1}+\ldots +a_n$ with $a_i\in \m$ for every $i\geq 0$ and $a_n\notin {\m}^2$. At first, we claim that$(p,f(X))$ is an $R[X]$-regular sequence. To see this, let $0\neq\overline{g(X)}=\overline{r_0} + \overline{r_1}X +\ldots +\overline{r_t}X^t\in (R/p)[X]$ with $\overline{r_t}\neq 0$. Now $f(X)\overline{g(X)}=\overline{f(X)}\overline{g(X)}=0$ implies $\overline{r_t}X^{t+n}=0$. Thus $\overline{r_t}=0$ and we have a contradiction. Thus $(p, f(X))$ is a $R[X]$-regular sequence. Using (b) of (12.2) Discussion of \cite{Ho} we can see that $(f(X),p)$ is also $R[X]$-regular sequences. %Similarly $(p,f(X))$ as well as $(f(X),p)$ are $H^{i}_{I}(R[X])$-regular sequences. In fact, previous two statements are also true when $(p,f(X))$ is also weakly regular sequence. 
Set $R[X]=R'$ and $f(X)=f$. 

Consider the following diagram of short exact sequences where every row and column is exact.

\[
\CD
@. 0@. 0@. 0@. @. \\
@. @VVV @VVV @VVV @.\\
0 @>>>R'@>f>>R'@>>>R'/fR'@>>>0 @.\\
@. @Vp VV @VVp V @VVp V @. \\
0 @>>>R'@>f>>R'@>>>R'/fR'@>>>0 @.\\
@. @VVV @VVV @VVV @. \\
0@>>>R'/pR'@>f>>R/pR'@>>>R'/(p,f)R'@>>>0 @.\\
@. @VVV @VVV @VVV @.\\
@. 0@. 0@. 0@. @. 
\endCD
\]

This above diagram yields the following diagram of long exact sequences where all rows and columns are exact.

\[
\CD
@. @. @. @. @. H^i_{I}(R')@.\\
@.@VVV @VVV @VVV @VVV @Vp VV \\
@. H^{i-1}_{I}(R') @>f>>H^{i-1}_{I}(R')@>\psi_{i-1}>>H^{i-1}_{I}(R'/fR')@>>>H^i_{I}(R')@>f>>H^i_{I}(R')@.\\
@. @V\pi_{i-1} VV @V\pi_{i-1} VV @V\alpha_{i-1} VV @V\pi_{i} VV @V\pi_{i} VV\\
@. H^{i-1}_{I}(R'/pR') @>f>>H^{i-1}_{I}(R'/pR')@>\phi_{i-1}>>H^{i-1}_{I}((R'/(p,f)R')@>>>H^i_{I}(R'/pR')@>f>>H^i_{I}(R'/pR')@.\\
@. @VVV @VVV @VVV @VVV @. \\
@. H^{i}_{I}(R') @>f>>H^{i}_{I}(R')@>>>H^{i}_{I}(R'/fR')@>>>H^{i+1}_{I}(R') @. @.\\
@. @Vp VV @Vp VV @Vp VV @Vp VV @. \\
@. H^{i}_{I}(R') @>f>>H^{i}_{I}(R')@>>>H^{i}_{I}(R'/fR')@>>>H^{i+1}_{I}(R') @. @.\\
@. @VVV @VVV @VVV @VVV @. 
\endCD
\]

From first paragraph of the proof, we already have $H^i_{I} (R'/pR')= 0$. Thus in the above diagram $\phi_{i-1}$ is surjective. Let $x$ be in $\Ker{(H^{i}_{I}(R'/fR')\stackrel{p}{\rightarrow}H^{i}_{I}(R'/fR'))}$. Then there exists some $y\in H^{i-1}_{I}(R'/(p, f)R')$ such that $x$ is the image of $y$ under the map $H^{i-1}_{I}(R'/(p, f)R')\rightarrow H^{i}_{I}(R'/fR')$. Since $\phi_{i-1}$ is surjective, $y= \phi_{i-1}(z)$ for some $z\in H^{i-1}_{I}(R'/pR')$. Now, if we come down to $H^{i}_{I}(R')$ via the map $H^{i-1}_{I}(R'/pR')\rightarrow H^{i}_{I}(R')$, due to injectivity of the multiplication map by $p$, the image of $z$ in $H^{i}_{I}(R')$ is zero. Since every square in the diagram is commutative, we get that $x= 0$. So, $p$ is a nonzero divisor of $H^{i}_{I}(R'/fR')= H^{i}_{IR[X]}(R[X]/fR[X])=H^{i}_{I(R[X]/fR[X])}(R[X]/fR[X])=H^{i}_{J}(S)$. Now using Theorem 1.2 of \cite{Nu1} we conclude. 

%Since $\pi_{i-1}$ is surjective, from exactness we have $f\circ\pi_{i-1}$ is surjective on $\Ker(\phi_{i-1})$. Along with the injectivity of $H^i_{I}(R')\stackrel{f}{\rightarrow}H^i_{I}(R')$, we also claim the injectivity of $H^i_{I}(R'/pR') \stackrel{f}{\rightarrow}H^i_{I}(R'/pR')$. To see this consider an element $x\in H^i_{I}(R'/pR')$ such that $fx=0$. Since $\pi_i$ is surjective there exists some $z\in H^{i}_{I}(R')$ such that $\pi_i(z)=x$. From commutativity of the square in the top-right corner, we get $fz= py$ since $\pi_i(fz)=0$. Since $(p,f)$ is a weakly $H^{i}_{I}(R')$-regular sequence, there exists $z'\in H^{i}_{I}(R')$ such that $z= pz'$. Thus $x=\pi_i(z)= \pi_i(pz')=0$ from exactness. Thus, we get the the diagram below where rows are exact.

%\[
%\CD
%@.H^{i-1}_{I}(R') @>f>>H^{i-1}_{I}(R')@>\psi_{i-1}>>H^{i-1}_{I}(R'/fR')@>>>0\\
%@. @V{f\circ\pi_{i-1}} VV @VV\pi_{i-1} V @V\alpha_{i-1} VV \\
%0@>>>\Ker(\phi_{i-1}) @>>>H^{i-1}_{I}(R'/pR')@>\phi_{i-1}>>H^{i-1}_{I}((R'/(p,f)R')@>>>0 @.
%\endCD
%\]

 %Thus using snake lemma we have $\Coker(\alpha_{i-1})=0$ since $\Coker(\pi_{i-1})=0$, which implies that $\alpha_{i-1}$ is surjective as well. This implies that $p$ is a nonzero divisor of $H^{i}_{IR[X]}(R[X]/fR[X])=H^{i}_{I(R[X]/fR[X])}(R[X]/fR[X])=H^{i}_{J}(S)$. Now using Theorem 1.2 of \cite{Nu1} we conclude. 
\end{proof}

Now, we state our first main result.

\begin{thm}
Let $S$ be a ramified regular local ring of dimension $d$ of mixed characteristic $p> 0$ and consider a set of regular system of parameters where $p$ is not in the ideal generated by any $d-1$ of them. Let $J\subset S$ be an ideal generated by monomials formed out of $d-1$ such regular system of parameters. Assume that, for a polynomial ring in $d-1$ indeterminates over the prime field of characteristic $p$, monomials similar to $J$ which are now formed out of $d-1$ indeterminates, generates a perfect ideal. Then for every $i\geq 0$ except $i= \height{J}+1, \height{J}+2$, $\Ass_{S} H^{i}_{J}(S)$ is finite.
\end{thm}

\begin{proof}
Since completion is a faithfully flat extension, using Lemma 2.1, it is sufficient to prove the theorem for complete ramified regular local ring. So, we can assume that $S$ is complete. Fix a regular system of parameters $x_1, x_2,\ldots, x_d$ in $S$. Let $J$ be an ideal generated by $d-1$ number of these regular system of parameters $x_2,\ldots, x_d$. %If each of the monomials is of total degree $n= 1$, then we get $J$ is generated by regular sequence and the result is trivial (Explain and elaborate more....). Thus we assume that total degree $n$ of each monomial $n\geq 2$. For convenience we choose $x_2,\ldots, x_d$ as the desired $d-1$ generators. 
Let $R_0$ be its coefficient ring. From hypothesis, we can get a system of parameters $p, x_2,\ldots, x_d$ in $S$. It turns out that $R_0[[x_2,\ldots, x_d]]$ is an unramified regular local ring where $x_2,\ldots, x_d$ are algebraically independent on $R_0$, i.e. we can think $R_0[[x_2,\ldots, x_d]]$ as power series ring in $d-1$ variables (see part of the proof of (iii) of Theorem 29.4 and (ii) of Theorem 29.8 of \cite{CRTofM}). Set $R= R_0[[x_2,\ldots, x_d]]$ and there exists Eisenstein polynomial $f(X)$, (which is actually minimal polynomial of $x_1$ over $R$) such that $S= R[X]/f(X)$ (see proof of (ii) of Theorem 29.8 of \cite{CRTofM}). 

Set $I= J\cap R[X]$. At first we would like to show that $\height{I}- \height{J}= 1$ and for that we can proceed as follows: Since $S= R[X]/f(X)$, there exists a maximal ideal $\m \in R[X]$ such that $S= R[X]/f(X)= R[X]_{\m}/f(X)R[X]_{\m}$. Since $I= J\cap R[X]$, we get $R[X]/I\cong S/J$. This gives $R[X]/I$ is local and $I\subset \m$. Moreover we have $f(X)\in I$. Since $S$ is regular local ring, $\dim{S}= \height{J}+ \dim{S/J}$. In the similar way, $\dim{R[X]_{\m}}= \height{IR[X]_{\m}}+ \dim{R[X]_{\m}/IR[X]_{\m}}$. Now every minimal primes over $I$ is inside $\m$, otherwise $R[X]/I$ should have more than one maximal ideal. This gives $\height{IR[X]_{\m}}= \height{I}$ and our last equation becomes $\dim{R[X]_{\m}}= \height{I}+ \dim{R[X]/I}$. On the other hand $\dim{R[X]_{\m}}= \height{f(X)R[X]_{\m}}+ \dim{R[X]_{\m}/f(X)R[X]_{\m}}$. Since $R[X]_{\m}$ is a domain, $\height{f(X)R[X]_{\m}}= 1$. Also we get $\dim{S}= \dim{R[X]/f(X)}= \dim{R[X]_{\m}/f(X)R[X]_{\m}}= \dim{R[X]_{\m}}- \height{f(X)R[X]_{\m}}$. Combining above results we get, $\dim{R[X]_{\m}}- \dim{S}= \height{I}- \height{J}= 1$.

Set $R_0[x_2,\ldots, x_d]= T$. Consider the ideal $\alpha\subset T$ such that $\alpha$ is generated by monomials similar to those which generates $J$. Since $p$ is an $T$-regular element, for short exact sequence $$0\rightarrow T\stackrel{p}{\rightarrow} T\rightarrow T/pT\rightarrow 0$$ we get the following long exact sequence of local cohomologies $$..H^{i-1}_{\alpha}(T/pT)\stackrel{\delta}\rightarrow H^{i}_{\alpha}(T)\stackrel{p}{\rightarrow}H^{i}_{\alpha}(T)\rightarrow H^{i}_{\alpha}(T/pT) \rightarrow ..$$ It is to be noted that, we have the natural map $T= R_0[x_2,\ldots, x_d]\rightarrow R_0[[x_2,\ldots, x_d]][X]=R[X]\rightarrow S$, such that $\alpha S= J$ where $\alpha R[X]\subset I$. 

Consider polynomial ring $(\ZZ/p\ZZ)[x_2,\ldots,x_d]$ of $d-1$ variables over prime field $\ZZ/p\ZZ$ and let $\beta$ be the ideal generated by monomials similar to those which generates $J$. From assumption, $\beta$ is perfect. Set $K= R_0/pR_0$. Clearly $(\ZZ/p\ZZ)[x_2,\ldots,x_d]\rightarrow (R_0/pR_0)[x_2,\ldots,x_d]= T/pT$ is faithfully flat extension where $(\alpha)(T/pT)= (\beta)(T/pT)$. It is easy to see that due to flatness $(\beta)(T/pT)$ i.e. $(\alpha)(T/pT)$ is perfect ideal in $T/pT$. Now From Proposition 19 of \cite{HE}, we get $T/(p+ \alpha)T$ is Cohen-Macaulay. From Proposition 4.1 of section III of \cite{PS}, we get $H^i_{\alpha}(T/pT)= 0$ for every $i\geq 0$, except $i= \height{\alpha}$. This gives that multiplication by $p$ on $H^i_{\alpha}(T)$ is an isomorphism for every $i\geq 0$ except $i= \height{\alpha}, \height{\alpha}+1$.

Next, we observe the following: For faithfully flat extension of Noetherian rings $A\rightarrow B$, let $\alpha$ be an ideal of $A$. Then $\grade{(\alpha, A)}= \grade{(\alpha B, B)}$. Let $a_1,\ldots,a_n$ be the maximal length of $A$-regular sequence which is inside $\alpha$. It is easy to observe $\grade{(\alpha, A)}\leq \grade{(\alpha B, B)}$. Let there be an element $a_{n+1} b\in \alpha B$ such that $a_1,\ldots,a_n, a_{n+1} b$ is a $B$-regular sequence in $\alpha B$. Let $a_{n+1}a\in (a_1,\ldots,a_n)$. This implies $a_{n+1}ba\in (a_1,\ldots,a_n)B$. This gives $a\in (a_1,\ldots,a_n)B\cap A= (a_1,\ldots,a_n)$. Moreover $A/(a_1,\ldots,a_{n+1})A\neq 0$. Thus $a_1,\ldots,a_{n+1}$ be the maximal length of $A$-regular sequence which is inside $\alpha$, but this is a contradiction. Thus $\grade{(\alpha, A)}= \grade{(\alpha B, B)}$.

Now, $T\rightarrow R[X]$ is faithfully flat extension. So tensoring with last long exact sequence of homology with $R[X]$ we get, $$..H^{i-1}_{\alpha R[X]}(R[X]/pR[X])\stackrel{\delta}\rightarrow H^{i}_{\alpha R[X]}(R[X])\stackrel{p}{\rightarrow}H^{i}_{\alpha R[X]}(R[X])\rightarrow H^{i}_{\alpha R[X]}(R[X]/pR[X]) \rightarrow ..$$ Moreover using results of above paragraph we get $\grade{(\alpha, T)}= \grade{(\alpha R[X], R[X])}$. Since we are in Cohen-Macaulay ring, height and grade of any ideal coincides. Thus we get, $H^i_{\alpha R[X]}(R[X]/pR[X])= 0$ for every $i\geq 0$, except $i= \height{\alpha R[X]}$. Moreover this gives that multiplication by $p$ on $H^i_{\alpha R[X]}(R[X])$ is an isomorphism every $i\geq 0$ except $i= \height{\alpha R[X]}, \height{\alpha R[X]}+1$.

Here $\alpha R[X]\subset R[X]$ such that $(\alpha R[X])S= J$. From above paragraph of the proof we have that multiplication by $p$ on $H^i_{\alpha R[X]}(R[X])$ is an isomorphism every $i\geq 0$ except $i= \height{\alpha R[X]}, \height{\alpha R[X]}+1$. Moreover from long exact sequence we get $p$ is a nonzero divisor of $H^{i+1}_{\alpha R[X]}(R[X])$ for every $i\geq 0$ except $i= \height{\alpha R[X]}, \height{\alpha R[X]}+1$. Thus from Proposition 3.1 we get that, for every $i\geq 0$ except $i= \height{\alpha R[X]}, \height{\alpha R[X]}+1$, $\Ass_S H^i_J(S)$ is finite.

To finish the proof, we need to show that $\grade (\alpha R[X], R[X])= \grade (I, R[X])$. We argue in the following way: Here $R[X]/f(X)$ is local, so there exists only one maximal ideal $\m$ which contains $f(X)$ and every primes which contains $f(X)$. Same is true for $I$. Thus $\grade (\alpha R[X], R[X])= \grade (\alpha R[X]_{\m}, R[X]_{\m})$ and similar is true for $I$. So we can pass to local ring $R[X]_{\m}$. Let $\grade (\alpha R[X]_{\m}, R[X]_{\m})= s$ and $\grade (IR[X]_{\m}, R[X]_{\m})= r$. Assume there exist two sets of sequences of maximal $R[X]_{\m}$-regular elements $\{z_1, z_2,\ldots, z_s\}$ and $\{f(X), y_2,\ldots, y_r\}$ in ideals $\alpha R[X]_{\m}$ and $IR[X]_{\m}$ respectively. We can change the regular sequence $\{z_1, z_2,\ldots, z_s\}$ to $\{f(X), z_2,\ldots, z_s\}$ where $f(X)$ may not be inside $\alpha R[X]_{\m}$. Passing to the ring $S= R[X]_{\m}/f(X)R[X]_{\m}$ we find that, $\alpha R[X]_{\m}S= IR[X]_{\m}S= J$ and we get two sets of $R[X]_{\m}/f(X)R[X]_{\m}$-regular sequences $\{\bar{z_2},\ldots, \bar{z_s}\}$ and $\{\bar{y_2},\ldots, \bar{y_r}\}$ in $J$. We claim that both are maximal. 

To prove this claim, let there be $y_{r+1}\in IR[X]_{\m}\subset R[X]_{\m}$ such that $\bar{y_{r+1}}$ is $S/(\bar{y_2},\ldots, \bar{y_r})S$-regular element. For $z\in R[X]_{\m}$, let $y_{r+1}z\in (f(X), y_2,\ldots,y_r)R[X]_{\m}$. This gives $\bar{y_{r+1}}\bar{z}\in (\bar{y_2},\ldots,\bar{y_r})S$ which in turn implies that $z\in (f(X), y_2,\ldots,y_r)R[X]_{\m}$. This contradicts the fact that $\{f(X), y_2,\ldots, y_r\}$ is a maximal $R[X]_{\m}$-regular sequence in $IR[X]_{\m}$. To see that $\{\bar{z_2},\ldots, \bar{z_s}\}$ is a maximal $S$-regular sequence, let there exists $z_{s+1}= a + bf(X)$ with $a\in \alpha R[X]$ and $b\notin f(X)R[X]$, such that $\bar{z_{s+1}}$ is $S/(\bar{z_2},\ldots, \bar{z_s})S$-regular sequence. At first, we observe that $a$ can not be zero. Otherwise, $f(X)$ and $bf(X)$ becomes part of regular sequence and $bf(X).1\in f(X)R[X]_{\m}$ but $1\notin f(X)R[X]_{\m}$. For some $z\in R[X]_{\m}$, assume $az\in (f(X), z_2,\ldots, z_s)R[X]_{\m}$. This gives $(a+ bf(X))z\in (f(X), z_2,\ldots, z_s)$. Thus passing to the ring $S$, we get $\bar{z}\in (\bar{z_2},\ldots, \bar{z_s})S$. Thus $z\in (f(X), z_2,\ldots, z_s)R[X]_{\m}$. Thus $f(X), z_2,\ldots, z_s, a$ forms $R[X]_{\m}$-regular sequence. Again replacing $f(X)$ by $z_1$, we get $(z_1, z_2,\ldots, z_s, a) R[X]\subset\alpha R[X]$ forms $R[X]$-regular sequence which contradicts the maximality of $R[X]$-regular sequence $z_1, z_2,\ldots, z_s$ in $\alpha R[X]_{\m}$. Thus both the sets of $R/f(X)$-regular sequences $\bar{z_2},\ldots, \bar{z_s}$ and $\bar{y_2},\ldots, \bar{y_r}$ in $J$ are maximal and $r= s$. Thus $\height{\alpha R[X]}= \height{I}$. 

From the second paragraph of the proof, we notice that $\height{I}= \height{J}+ 1$. Thus for every $i\geq 0$, except $i= \height{J}+1, \height{J}+2$, $\Ass_S H^i_J(S)$ is finite. Thus we conclude.
\end{proof}

We observe our next important result as an application of above theorem in the following corollary.

\begin{cor}
Consider ramified regular local ring $S$ of dimension $d$ of mixed characteristic $p> 0$ and let there be positive integers $m$ and $n$ such that $d-1= mn$. Consider a set of regular system of parameters where $p$ is not in the ideal generated by any $d-1$ of them. If $J$ is an ideal, generated by size $t$ minors of $m\times n$ matrix formed by $d-1$ such regular system of parameters, then for every $i\geq 0$, except $i= \height{J}+1, \height{J}+2$, $\Ass_{S} H^{i}_{J}(S)$ is finite.
\end{cor}

\begin{proof}
Consider polynomial ring $K[\{X_{ij}:m\geq i\geq 1, m\geq i\geq 1\}]= K[X_{ij}]$ (for brevity) over field $K$, in matrix of indeterminates and let $I_t (X)K[X_{ij}]$ is an ideal generated by size $t$ minor of matrix $(X_{ij})$. From \cite{HE}, we get that for $I_t (X)K[X_{ij}]$ is perfect and result is then immediate from above theorem. 

%Now we adopt the same notation 
\end{proof}

\begin{rem}
In Theorem 3.2 above, it is to be noted that minimal polynomial $f(X)$ of $x_1$ in the ring $R[X]$, is not even inside the radical of $\alpha R[X]$. Let $f(X)= X^n+ b_1 X^{n-1} +\ldots+ b_n$. Raising to a power of integer $t$, we get $(f(X))^t= X^{nt}+\dots+ b^t_n$. If $(f(X))^t\in \alpha R[X]$ then $1\in \alpha$, which is a contradiction. %Passing to the ring $S$, we get $f(x_1)= 0$. This gives $b^t_n\in R\cap (x_1)S= 0$, which implies that $b_n= 0$. But being minimal polynomial of $x_1$, $b_n\neq 0$ and this is a contradiction.
\end{rem}

\begin{example}
Here, we give examples of local cohomology modules such that $p$ is a nonzero divisor of it, or more generally, multiplication by $p$ on local cohomology modules, is an isomorphism. 

(a) In section 4 of \cite{Bh}, there are examples of local cohomology modules such that $p$ is a nonzero divisor of it.

(b) In the proof of Theorem 3.2, we have already seen certain examples of ideals, so that multiplication by $p$, on the local cohomology modules with support in those ideals, is an isomorphism. 

(c) Here we present another source of examples such that upon multiplication by $p$, on the local cohomology modules, is an isomorphism: Consider an ideal $J\subset pS$. Since $p$ is an $S$-regular element, for short exact sequence $$0\rightarrow S\stackrel{p}{\rightarrow} S\rightarrow S/pS\rightarrow 0$$ we get the following long exact sequence of local cohomologies $$..\rightarrow H^{i-1}_{J}(S/pS)\stackrel{\delta}\rightarrow H^{i}_{J}(S)\stackrel{p}{\rightarrow}H^{i}_{J}(S)\rightarrow H^{i}_{J}(S/pS) \rightarrow ..$$ Since expansion of $J$ in $S/pS$ is a zero ideal, for every $i\geq 1$, $H^{i}_{J}(S/pS)= 0$. Thus, above long exact sequence reduces to $$0\rightarrow H^{0}_{J}(S)\stackrel{p}{\rightarrow}H^{0}_{J}(S)\rightarrow H^{0}_{J}(S/pS) \stackrel{\delta}\rightarrow H^{1}_{J}(S)\stackrel{p}{\rightarrow}H^{1}_{J}(S)\rightarrow 0$$ and $$0\rightarrow H^{i}_{J}(S)\stackrel{p}{\rightarrow}H^{i}_{J}(S)\rightarrow 0$$ for every $i\geq 2$. 
\end{example}

Now we present our third main result of this paper in the following theorem. It proves that, there always exists a ramified regular local ring, with a countable collection of ideals, such that we get a countable collection of local cohomology modules with support in those ideals, having finite set of associated primes.

\begin{thm}
Let $\{I_{\lambda}\}$ be a countable collection of ideals of an unramified regular local ring $R$ in mixed characteristic $p>0$, such that for every $\lambda$, $p$ is a non zero divisor of $H^i_{I_{\lambda}}(R)$ and $H^{i+1}_{I_{\lambda}}(R)$, for every $i\in C_{\lambda}\subset \NN$ and for every $I_{\lambda}$, where $C_{\lambda}$ is a countable subset of $\NN$. Then, there exists a ramified regular local ring $S$ of mixed characteristic $p> 0$, which is a homomorphic image of $R$, so that for expanded ideal $J_{\lambda}= I_{\lambda}S$ of $S$, and for corresponding $C_{\lambda}\subset \NN$, $\Ass_S H^i_{J_{\lambda}}(S)$ is finite for every $i\in C_{\lambda}$ and this is true for every $\lambda$. 
\end{thm}

\begin{proof}
At first we observe that, we can always include $I_{\lambda}= 0$ in the countable collection of ideals as stated in the hypothesis along with corresponding $C_{\lambda}= \NN$. Let $(R,\m)$ be an unramified regular local ring in mixed characteristic $p>0$.  Since $p$ is an $R$-regular element, for short exact sequence $$0\rightarrow R\stackrel{p}{\rightarrow} R\rightarrow R/pR\rightarrow 0$$ we get the following long exact sequence of local cohomologies $$..\rightarrow H^{i-1}_{I_{\lambda}}(R/pR)\stackrel{\delta}\rightarrow H^{i}_{I_{\lambda}}(R)\stackrel{p}{\rightarrow}H^{i}_{I_{\lambda}}(R)\rightarrow H^{i}_{I_{\lambda}}(R/pR) \rightarrow ..$$ Since $p$ be a non zero divisor of $H^i_{I_{\lambda}} (R)$ and  $H^{i+1}_{I_{\lambda}} (R)$ for every $i\in C_{\lambda}$ and for every $I_{\lambda}$, from above long exact sequence, we get $H^i_{I_{\lambda}} (R/pR)= H^i_{I_{\lambda}} (R)/pH^i_{I_{\lambda}} (R)$.

We can assume $R$ to be a complete regular local ring. From \cite{Pu2}, we know that for every ideal $I_{\lambda}$ of $R$, and for every $i\geq 0$, the set $\Ass_R H^i_{I_{\lambda}} (R)$ is countable. Since $p$ is regular, same is true for $\Ass_R H^i_{I_{\lambda}} (R)/pH^i_{I_{\lambda}} (R)$. For any $R$-module $M$, set $\overline{\Ass}_R M= \Ass_R M- \{\m\}$. Now $A= \bigcup_{\lambda} (\cup_{i\in C_{\lambda}}\overline{\Ass}_R H^i_{I_{\lambda}} (R))\cup (\cup_{i\in C_{\lambda}} \overline{\Ass}_R H^i_{I_{\lambda}} (R)/pH^i_{I_{\lambda}} (R))$ is countable. From \cite{Bu} we get that there exists an element $g\in \m$ such that $g$ does not belong to any element of $A$. From construction, $(p, h)$ is a $H^i_{I_{\lambda}} (R)$-regular sequence. Set $h= g^2$ and thus $(p, h)$ also becomes a $H^i_{I_{\lambda}} (R)$-regular sequence where $h\in {\m}^2$. %Since $h$ is also a non zero divisor of $H^i_{I_{\lambda}} (R)$ for every $i\geq 0$ and for every $I_{\lambda}$, by (b) of (12.2) Discussion of \cite{Ho}, we find $(h, p)$ is also a $H^i_{I_{\lambda}} (R)$-regular sequence. Similar argument shows that $(h, p)$ is also a $R$-regular sequence. For $I_{\lambda}= 0$, we get $\overline{\Ass}_R H^0_{0} (R/pR)= \overline{\Ass}_R R/pR$.

Now consider the sequence $(p, h-p)$. Let $(h-p)x= py$ for $x, y\in H^{i}_{I_{\lambda}}(R)$ for some $i\in C_{\lambda}$ and for some ideal $I_{\lambda}$. Then $hx= p(x+ y)$ and since $(p, h)$ is an $H^i_{I_{\lambda}} (R)$-regular sequence we get that $(p, h-p)$ is also an $H^i_{I_{\lambda}} (R)$-regular sequence. In fact, for $I_{\lambda}= 0$, we get $\overline{\Ass}_R H^0_{0} (R/pR)= \overline{\Ass}_R R/pR$. Thus, from second paragraph of the proof we get that $(p, h)$ as well as $(p, h-p)$ are $R$-regular sequence.

Set $h-p= f$. Since $h\in {\m}^2$, we find $f\in \m- {\m}^2$. Thus, it is part of a regular system of parameters. So, $R/fR$ is again a regular local ring such that $p\in {\m}^2$. Moreover, $h\notin pR$, otherwise, $h= h.1 =pr$ implies $1\in pR$ since $(p,h)$ are regular sequence. Thus, it is a contradiction. This shows $p$ is not a multiple of $h-p$. So, $R/fR$ is a ramified regular local ring of mixed characteristic $p\geq 0$. Moreover, each of $h$ and $f$ are also nonzero divisor in $R$. Thus using (b) of (12.2) Discussion of \cite{Ho}, we find $(h, p)$ as well as $(f, p)$ are also a $R$-regular sequence. So, we can consider the following diagram of short exact sequences where every row and column is exact.

\[
\CD
@. 0@. 0@. 0@. @. \\
@. @VVV @VVV @VVV @.\\
0 @>>>R@>f>>R@>>>R/fR@>>>0 @.\\
@. @Vp VV @VVp V @VVp V @. \\
0 @>>>R@>f>>R@>>>R/fR@>>>0 @.\\
@. @VVV @VVV @VVV @. \\
0@>>>R/pR@>f>>R/pR@>>>R/(p,f)R@>>>0 @.\\
@. @VVV @VVV @VVV @.\\
@. 0@. 0@. 0@. @. \endCD
\]

This above diagram yields the following diagram of long exact sequences where every rows and columns are exact.

\[
\CD
@. @. @. @. @. H^i_{I_{\lambda}}(R)@.\\
@.@VVV @VVV @VVV @VVV @Vp VV \\
@. H^{i-1}_{I_{\lambda}}(R) @>f>>H^{i-1}_{I_{\lambda}}(R)@>\psi_{i-1}>>H^{i-1}_{I_{\lambda}}(R/fR)@>>>H^i_{I_{\lambda}}(R)@>f>>H^i_{I_{\lambda}}(R)@.\\
@. @V\pi_{i-1} VV @V\pi_{i-1} VV @V\alpha_{i-1} VV @V\pi_{i} VV @V\pi_{i} VV\\
@. H^{i-1}_{I_{\lambda}}(R/pR) @>f>>H^{i-1}_{I_{\lambda}}(R/pR)@>\phi_{i-1}>>H^{i-1}_{I_{\lambda}}((R/(p,f)R)@>>>H^i_{I_{\lambda}}(R/pR)@>f>>H^i_{I_{\lambda}}(R/pR)@.\\
@. @VVV @VVV @VVV @VVV @. \\
@. H^{i}_{I_{\lambda}}(R) @>f>>H^{i}_{I_{\lambda}}(R)@>>>H^{i}_{I_{\lambda}}(R/fR)@>>>H^{i+1}_{I_{\lambda}}(R) @. @.\\
@. @Vp VV @Vp VV @Vp VV @Vp VV @. \\
@. H^{i}_{I_{\lambda}}(R) @>f>>H^{i}_{I_{\lambda}}(R)@>>>H^{i}_{I_{\lambda}}(R/fR)@>>>H^{i+1}_{I_{\lambda}}(R) @. @.\\
@. @VVV @VVV @VVV @VVV @. 
\endCD
\]

In the above diagram, $\phi_{i-1}$ is surjective since $f$ is $H^i_{I_{\lambda}} (R/pR)= H^i_{I_{\lambda}} (R)/p H^i_{I_{\lambda}} (R)$-regular. Let $x$ be in $\Ker{(H^{i}_{I_{\lambda}}(R/fR)\stackrel{p}{\rightarrow}H^{i}_{I_{\lambda}}(R/fR))}$. Then there exists some $y\in H^{i-1}_{I_{\lambda}}(R/(p, f)R)$ such that $x$ is the image of $y$ under the map $H^{i-1}_{I_{\lambda}}(R/(p, f)R)\rightarrow H^{i}_{I_{\lambda}}(R/fR)$. Since $\phi_{i-1}$ is surjective, $y= \phi_{i-1}(z)$ for some $z\in H^{i-1}_{I_{\lambda}}(R/pR)$. Now, if we come down to $H^{i}_{I_{\lambda}}(R)$ via the map $H^{i-1}_{I_{\lambda}}(R/pR)\rightarrow H^{i}_{I_{\lambda}}(R)$, due to injectivity of the multiplication map by $p$, the image of $z$ in $H^{i}_{I_{\lambda}}(R)$ is zero. Since every square in the diagram is commutative, we get that $x= 0$. So, $p$ is a nonzero divisor of $H^{i}_{I_{\lambda}}(R/fR)$. For ramified regular local ring $R/fR$, set $S= R/fR$. Then, using Theorem 1.2 of \cite{Nu1}, we get $\Ass_S H^{i}_{J_{\lambda}}(S)$ is finite.

\end{proof}

%{\textbf{Acknowledgement:}}\newline

%I would like to thank Tony J. Puthenpurakal for his invaluable comments and suggestions.
%%%%%%%%%%%%%%%%%%%%%%%%%%%%%%%%%%%%%%%%%%%%%%%%%%%%%%%%%%%%%%%%%%%%%%%%%%%%%%%%%%%%%%%%%%%%%%%%%%%%%%%%%%%%%%%%%%%%%%%%%%%%%%%%%%%%%%%%
%%%%%%%%%%%%%%%%%%%%%%%%%%%%%%%%%%%%%%%%%%%%%%%%%%%%%%%%%%%%%%%%%%%%%%%%%%%%%%%%%%%%%%%%%%%%%%%%%%%%%%%%%%%%%%%%%%%%%%%%%%%%%%%%%%%%%%%%%%%%%%%%%%%%%%%%%%%%%%%%%%%%%%%%%%%%%%%%%%%%%%%%%%%%%%%%%%%%%%%%%%%%%%%%%%%%%%%%%%%%%%%%%%%%%%%%%%%%%%%%%%%%%%%%%%%%%%%%%%%%

\end{document}